\documentclass[a4paper,11pt]{amsart}

\usepackage{enumerate, booktabs}
\usepackage{amsmath, amsfonts, amssymb, amsthm}
\usepackage[all]{xy}
\usepackage[height=22.5cm, width=15.5cm]{geometry}
\usepackage{mathpazo}

\numberwithin{equation}{section}

\theoremstyle{plain}
\newtheorem{thm}{Theorem}[section]
\newtheorem{prop}[thm]{Proposition}
\newtheorem{lem}[thm]{Lemma}

\newtheorem*{thm*}{Theorem}

\theoremstyle{definition}

\theoremstyle{remark}
\newtheorem*{rem}{Remark}
\newtheorem*{ex}{Example}

\renewcommand{\leq}{\leqslant}
\renewcommand{\geq}{\geqslant}

\newcommand{\mbb}[1]{\mathbb{#1}}

\newcommand{\wh}[1]{\widehat{#1}}
\newcommand{\ol}[1]{\overline{#1}}
\newcommand{\lie}[1]{{\mathfrak{#1}}}
\newcommand{\abs}[1]{\lvert #1\rvert}
\newcommand{\norm}[1]{\lVert #1\rVert}

\newcommand{\vf}[1]{\frac{\partial}{\partial #1}}

\DeclareMathOperator{\im}{Im}

\DeclareMathOperator{\End}{End}
\DeclareMathOperator{\Aut}{Aut}

\DeclareMathOperator{\Sym}{Sym}
\DeclareMathOperator{\tr}{Tr}


\title[On quotients of bounded homogeneous domains]{On quotients of bounded 
homogeneous domains by unipotent discrete groups}

\author{Christian Miebach}
\address{Univ.~Littoral C\^ote d'Opale, UR 2597, LMPA, Laboratoire de 
Math\'ematiques Pures et Appliqu\'ees Joseph Liouville, F-62100 Calais, France}
\email{christian.miebach@univ-littoral.fr}

\begin{document}

\begin{abstract}
We show that the quotient of any bounded homogeneous domain by a unipotent 
discrete group of automorphisms is holomorphically separable. Then we give a 
necessary condition for such a quotient to be Stein and prove that in some cases 
this condition is also sufficient.
\end{abstract}

\maketitle

\section{Introduction}

Given a Lie group of holomorphic transformations of a Stein space $X$, one would 
like to have a complex quotient space whose holomorphic functions are the 
invariant holomorphic functions on $X$ and which is again Stein. In the case 
that the Lie group is \emph{compact}, it is possible to average holomorphic 
functions over this group and to construct in this way an invariant-theoretic 
Stein quotient space, see~\cite{He}. On the other hand, when an infinite 
discrete group acts properly by holomorphic transformations on a Stein space, 
then the orbit space is again a complex space. However, in general there is no 
averaging method in order to construct holomorphic functions on this quotient 
space. In fact, it is not hard to find examples where the quotient space is 
compact, hence where all invariant holomorphic functions are constant.

There are numerous results in the literature where certain quotients of Stein 
manifolds by proper actions of infinite discrete groups are shown to be 
holomorphically separable or Stein. In~\cite{FabIan} the authors showed that 
quotients of the unit ball $\mbb{B}_n\subset\mbb{C}^n$ by proper 
$\mbb{Z}$-actions are Stein, which was then generalized to simply-connected 
bounded domains of holomorphy in~$\mbb{C}^2$ (\cite{MO}), to arbitrary 
bounded homogeneous domains in $\mbb{C}^n$ (\cite{Mie4}), and to 
Akhiezer-Gindikin domains (\cite{Vi}). Quotients of the unit ball and 
Akhiezer-Gindi\-kin domains by discrete groups that act cocompactly on a real 
form of these domains are shown to be Stein in~\cite[Proposition~6.4]{BS} 
and~\cite[Corollary~7]{BHH}, respectively. An analog result is shown in~\cite{L} 
for quotients of complex solvable Lie groups by discrete subgroups that act 
cocompactly on a real form having purely imaginary spectrum. Quotients of 
complex Olshanski semigroups by certain discrete groups were studied 
in~\cite{ABKr} and~\cite{Mie5}. In~\cite{Ch} actions of discrete groups on 
K\"ahler-Hadamard manifolds and their quotients are investigated. Recently, 
quotients of the unit ball by certain convex-cocompact discrete groups were 
studied from the viewpoint of complex-hyperbolic geometry in~\cite{DK}. 
In~\cite{MO2} Schottky group actions on the unit ball having Stein quotients are 
constructed. Most of these results strongly rely on Lie theory.

In this paper, we are concerned with the action of a \emph{unipotent} discrete 
group $\Gamma$ of holomorphic automorphisms of a bounded homogeneous domain. In 
order to state the main results, let $D\subset\mbb{C}^n$ be a bounded domain and 
recall that its automorphism group $\Aut(D)$ is a real Lie group that acts 
properly on $D$, see~\cite{Car}. Let $G$ be the connected component of $\Aut(D)$ 
that contains the identity. The domain $D$ is \emph{homogeneous} if $\Aut(D)$ 
and hence $G$ act transitively on it. Let us fix a base point $p_0\in D$. Its 
isotropy group $K:=G_{p_0}$ is a maximal compact subgroup of $G$ and there 
exists a decomposition $G=KR$ where $R$ is a simply-connected split solvable 
Lie group, see~\cite{Kan2} and~\cite{Vin4}. Consequently, $R$ is isomorphic to 
a semi-direct product $A\ltimes N$ where $N$ is the nilradical of $R$ and 
$A\cong(\mbb{R}^{>0})^r$.

Let $D\subset\mbb{C}^n$ be a bounded homogeneous domain and let $G=KAN$ be the 
decomposition of $G=\Aut^0(D)$ introduced above. A discrete subgroup of $G$ 
will be called \emph{unipotent} if it is conjugate to a subgroup of $N$. The 
research presented in this paper was motivated by the following result.

\begin{thm}\label{Thm:holsep}
Let $D\subset\mbb{C}^n$ be a bounded homogeneous domain and let $\Gamma$ be a 
unipotent discrete group of automorphisms of $D$. Then the complex manifold 
$D/\Gamma$ is holomorphically separable. 
\end{thm}

It is therefore natural to ask under which additional conditions on $\Gamma$ 
the quotient manifold $D/\Gamma$ is Stein. Suppose from now on that $\Gamma$ is 
a discrete subgroup of $N$. It is well known that the simply-connected nilpotent 
group $N$ admits a unique structure as a real-algebraic group such that its 
Zariski closed subgroups are precisely its connected Lie subgroups, 
see~\cite[Chapter~2.4.2]{Vin3}. Hence, consider the Zariski closure $N_\Gamma$ 
of $\Gamma$ in $N$. Then $N_\Gamma$ is a simply-connected nilpotent Lie group 
such that $N_\Gamma/\Gamma$ is compact.

We have the following necessary condition for $D/\Gamma$ to be Stein. Note that 
to the best of my knowledge it is not known whether this condition is also 
sufficient.

\begin{prop}\label{Prop:nec}
Let $\Gamma\subset N$ be a discrete subgroup and consider its Zariski closure 
$N_\Gamma$. If $D/\Gamma$ is Stein, then all $N_\Gamma$-orbits in $D$ are 
totally real.
\end{prop}

Since $R\cong A\ltimes N$ acts simply transitively on $D$, we can identify its 
Lie algebra $\lie{r}=\lie{a}\oplus\lie{n}$ with $T_{p_0}D=\mbb{C}^n$ and thus 
obtain a complex structure $j\in\End(\lie{r})$. Let $N_\Gamma$ be the 
Zariski closure of a discrete subgroup $\Gamma\subset N$ and let 
$\lie{n}_\Gamma$ be its Lie algebra. The orbit $N_\Gamma\cdot p_0$ is totally 
real if and only if $\lie{n}_\Gamma\cap j(\lie{n}_\Gamma)=\{0\}$, in which case 
we call $\lie{n}_\Gamma$ a \emph{totally real subalgebra} of $\lie{n}$. 
If $\lie{n}_\Gamma$ is totally real and if the real dimension of 
$\lie{n}_\Gamma$ coincides with the complex dimension of $D$, then we say 
that $\lie{n}_\Gamma$ is a \emph{maximal totally real subalgebra} of $\lie{n}$. 
We have the following sufficient criterion for $D/\Gamma$ to be Stein.

\begin{prop}\label{Prop:suff}
Let $\Gamma\subset N$ be a discrete subgroup and consider its Zariski closure 
$N_\Gamma$. If $\lie{n}_\Gamma$ is contained in a maximal totally real 
subalgebra of $\lie{n}$, then $D/\Gamma$ is Stein.
\end{prop}

The methods and results described above allow us to prove the main result of 
this paper, which in particular answers the question raised 
in~\cite[Remark~7.6]{Ch} in the negative.

\begin{thm}\label{Thm:main}
Let $D$ be the unit ball or the Lie ball and let $\Gamma$ be a unipotent 
discrete group of automorphisms of $D$ having Zariski closure $N_\Gamma \subset 
N$. Then $D/\Gamma$ is Stein if and only if $\lie{n}_\Gamma$ is totally real.
\end{thm}

The proof of Theorem~\ref{Thm:main} relies in part on the fact that for the 
unit ball and the Lie ball the necessary condition given in 
Proposition~\ref{Prop:nec} is indeed sufficient as well.

Let us outline the structure of this paper. In Section~\ref{Sect:gen} we review 
some parts of the structure theory of bounded homogeneous domains and prove 
Theorem~\ref{Thm:holsep} as well as Propositions~\ref{Prop:nec} 
and~\ref{Prop:suff}. Sections~\ref{Sect:unit} and~\ref{Sect:Lie} contain the 
proves of Theorem~\ref{Thm:main} for the unit ball and the Lie ball, 
respectively. In the last section~\ref{Sect:Siegel} we present an example that 
shows that Theorem~\ref{Thm:main} does \emph{not} hold true for arbitrary 
bounded homogeneous domains.

\subsection*{Acknowledgments}

I would like to thank Karl Oeljeklaus for helpful discussions. Financial support 
by the ANR-project QuaSiDy (ANR-21-CE40-0016-01) is gratefully acknowledged.

\section{Results for arbitrary bounded homogeneous domains}\label{Sect:gen}

We begin by reviewing the structure of the automorphism group of a bounded 
homogeneous domain, its unbounded realization as a Siegel domain, and its 
associated normal $j$-algebra. As main references we refer the reader 
to~\cite{Kan2} and~\cite{Pya}. We then show that the quotient of a bounded 
homogeneous domain by a unipotent discrete group $\Gamma$ of automorphisms is 
holomorphically separable, see Theorem~\ref{Thm:holsep}, and prove 
Propositions~\ref{Prop:nec} and~\ref{Prop:suff}.

\subsection{Bounded homogeneous domains and normal $j$-algebras}

Let $D\subset\mbb{C}^n$ be a bounded homogeneous domain with base point $p_0\in 
D$ and consider the decomposition $G=KR$ where $G=\Aut^0(D)$ and $K=G_{p_0}$. 
The elements of the Lie algebra $\lie{g}$ will be viewed as complete 
holomorphic vector fields on $D$. Since the group $R$ is split solvable, we 
have $R\cong A\ltimes N$ where $N$ is the nilradical of $R$. Moreover, the 
adjoint representation of the Abelian group $A$ on $\lie{n}$ is diagonalizable 
over $\mbb{R}$, compare~\cite[Proposition~2.8]{Kan2}.

It is well known that every bounded homogeneous domain $D\subset\mbb{C}^n$ is 
biholomorphic to a Siegel domain of the second kind $\wh{D}\subset\mbb{C}^n$, 
see~\cite{VGP}.

\begin{ex}
The unit ball $\mbb{B}_n\subset\mbb{C}^n$ is biholomorphically equivalent to 
\begin{equation*}
\wh{\mbb{B}}_n=\bigl\{(z,w)\in\mbb{C}\times\mbb{C}^{n-1};\ \im(z)-\norm{w}^2>0 
\bigr\},
\end{equation*}
the Lie ball $\mbb{L}_n\subset\mbb{C}^n$ is biholomorphic to
\begin{equation*}
\wh{\mbb{L}}_n=\bigl\{z\in\mbb{C}^n;\ \im(z_n)^2-\bigl(\im(z_1)^2+\dotsb+ 
\im(z_{n-1})^2\bigr)>0,\ \im(z_n)>0\bigr\},
\end{equation*}
and the Siegel disk $\mbb{S}_n\subset\Sym(n,\mbb{C})$ may be realized as 
Siegel's upper half-plane
\begin{equation*}
\wh{\mbb{S}}_n=\bigl\{Z\in\Sym(n,\mbb{C});\ \im(Z)\text{ is positive 
definite}\bigr\}.
\end{equation*}
In Helgason's notation, see~\cite[Table~V, Chapter~X]{H}, the unit ball, the Lie 
ball and the Siegel disk correspond to the Hermitian symmetric spaces of types 
$AIII(p=1,q=n)$, $BDI(p=2,q=n)$ and $CI$, respectively.
\end{ex}

The split-solvable group $R\cong A\ltimes N$ acts by affine transformations on 
$\wh{D}$. In particular, this $R$-action extends to the whole of $\mbb{C}^n$. 
It follows from the explicit realization of its affine automorphism group, see 
e.g.~\cite[Chapters~2 and~3]{Kan2}, that $N$ and hence $N^\mbb{C}$ act 
\emph{algebraically} by affine transformations on $\mbb{C}^n$ and that every 
transformation in $N^\mbb{C}$ has Jacobi determinant equal to $1$.

Since $R$ acts simply transitively on $D$, we may identify its Lie algebra 
$\lie{r}=\lie{a}\oplus\lie{n}$ with $T_{p_0}D=\mbb{C}^n$ and thus obtain an 
integrable complex structure $j\in\End(\lie{r})$.

In order to show that $(\lie{r},j)$ is a normal $j$-algebra we need to prove the
existence of a linear form $\lambda\in\lie{r}^*$ such that for all 
$x,y\in\lie{r}$ we have
\begin{equation*}
\lambda\bigl[j(x),j(y)\bigr]=\lambda[x,y]\quad\text{and}\quad
\lambda\bigl[j(x),x\bigr]>0\text{ if }x\not=0,
\end{equation*}
see part (III) of the definition in~\cite[p.~51]{Pya}.

In~\cite{Kan2} and~\cite{Pya} existence of this linear form 
$\lambda\in\lie{r}^*$ is deduced from~\cite{Kos}. Let us explain here how 
$\lambda$ can be constructed from the Bergman metric of $D$ via a moment map. 
(Recall that the Bergman metric of a bounded domain $D\subset\mbb{C}^n$ is an 
$\Aut(D)$-invariant K\"ahler metric on $D$. For its definition we refer the 
reader to~\cite[Chapter~4.10]{Kob}.) It is shown in~\cite{Is} that there exists 
an $R$-equivariant holomorphic embedding of $D$ into Siegel's upper half-plane 
$\wh{\mbb{S}}_n$. Since the domain $\mbb{S}_n$ is symmetric, the group 
$\Aut^0(\mbb{S}_n)$ is semisimple and thus there exists an equivariant moment 
map for the action of $\Aut^0(\mbb{S}_n)$ on $\mbb{S}_n$. Pulling back this 
moment map to $D$ we obtain an $R$-equivariant moment map $\mu\colon 
D\to\lie{r}^*$ with respect to the Bergman metric $\omega$ of $D$. Note that 
$\mu$ is a diffeomorphism onto its image, an open coadjoint orbit. Moreover, 
$\mu$ is a symplectomorphism with respect to the Kostant-Kirillov form on this 
coadjoint orbit, i.e., we have
\begin{equation}\label{Eqn:moment}
\omega_{p_0}\bigl(x(p_0),y(p_0)\bigr)=\mu(p_0)[x,y]
\end{equation}
for all $x,y\in\lie{r}$, see~\cite[Equations~(26.2) and~(26.8)]{GS}. As a 
consequence, we can take $\lambda:=\mu(p_0)\in\lie{r}^*$ in the definition 
of the normal $j$-algebra associated with $D$. For the structure theory of 
normal $j$-algebras we refer the reader to~\cite[Chapter~2]{Pya}.

\begin{rem}
For a connected closed subgroup $R'$ of $R$ the orbit $R'\cdot p_0$ is totally 
real if and only if $\lie{r}'\cap j(\lie{r}')=\{0\}$. In this case, we refer to 
$\lie{r}'$ as a totally real subalgebra of $\lie{r}$. Note that every isotropic 
orbit (with respect to the Bergman metric) is totally real. Moreover, every 
orbit of an Abelian subgroup of $R$ is totally real since $\mu(p_0)\bigl[j(x), 
x\bigr]>0$ holds for $x\not=0$. This last statement is also a consequence of 
the fact that $D$ is Kobayashi-hyperbolic, since an Abelian group of 
automorphisms that is not totally real would yield a non-constant holomorphic 
map from $\mbb{C}$ to $D$.
\end{rem}

Conversely, every abstract normal $j$-algebra is the normal $j$-algebra 
associated with some bounded homogeneous domain, see~\cite[Appendix]{Pya}.

\subsection{Quotients by unipotent discrete groups are holomorphically 
separable}

Let $D\subset\mbb{C}^n$ be a bounded homogeneous domain and let $\Gamma$ be a 
unipotent subgroup of $G=\Aut^0(D)$.

\begin{proof}[Proof of Theorem~\ref{Thm:holsep}]
Any unipotent discrete subgroup $\Gamma\subset G$ acts freely and properly on 
$D\times\mbb{C}$ by
\begin{equation*}
\gamma\cdot(z,t):=\bigl(\gamma(z),t\det d\gamma(z)^{-1}\bigr),
\end{equation*}
so that we have the quotient manifold $L:=D\times_\Gamma\mbb{C} 
:=(D\times\mbb{C})/\Gamma$. Moreover, the projection onto the first factor 
$D\times\mbb{C}\to D$ is $\Gamma$-equivariant and the induced map 
$L=D\times_\Gamma\mbb{C}\to D/\Gamma$ defines a holomorphic line bundle on 
$D/\Gamma$. A holomorphic section in $L$ corresponds to a $\Gamma$-equivariant 
holomorphic function from $D$ to $D\times\mbb{C}$. Holomorphic sections 
in the $k$-fold tensor product $L^{\otimes k}$ can be constructed via 
Poincar\'e series for $k\geq2$ and separate the points of $D/\Gamma$ in the 
following sense. For every pair of elements $p,q\in D/\Gamma$ with $p\not=q$ 
there exist $k\geq2$ (depending on $p$ and $q$) and a holomorphic section $s$ 
in $L^{\otimes k}$ such that $s(p)=0$ and $s(q)\not=0$, 
see~\cite[Lemma~3.4.1]{Pya}. We will finish the proof by showing that for a 
unipotent discrete subgroup $\Gamma$ the line bundle $L$ admits a non-vanishing 
holomorphic section. It then follows that $L$, and thus all of its powers, 
are holomorphically trivial.

In order to do so, let $\varphi\colon D\to\wh{D}$ be a $\Gamma$-equivariant 
biholomorphic map to a Siegel domain of the second kind on which $\Gamma$ acts 
by affine transformations having Jacobi determinant $1$. In other words, we 
have $\varphi\circ\gamma=\wh{\gamma}\circ\varphi$ where $\det 
d\wh{\gamma}(z)=1$ for all $z\in\wh{D}$. Then, the chain rule implies
\begin{align*}
\det d\varphi\bigl(\gamma(z)\bigr)&=\det d\gamma(z)^{-1}\cdot \det 
d(\varphi\circ \gamma)(z)\\
&=\det d\gamma(z)^{-1}\cdot\det d(\wh{\gamma}\circ\varphi)(z)=\det 
d\gamma(z)^{-1}\cdot\det d\varphi(z).
\end{align*}
Consequently, the holomorphic map $s\colon D\to D\times\mbb{C}^*$ given by 
$s(z,t)=\bigl(z,\det d\varphi(z)\bigr)$ is $\Gamma$-equivariant and thus 
defines a non-vanishing holomorphic section in $L$, as desired.
\end{proof}

\subsection{A necessary condition for Steinness}

In this subsection we prove Proposition~\ref{Prop:nec}.

An important ingredient for the proof is the following result of Loeb. Let 
$\Gamma\subset G$ be a discrete group of unipotent automorphisms of $D$. After 
conjugation we may suppose that $\Gamma$ is contained in $N$. Let $N_\Gamma$ be 
the real Zariski closure of $\Gamma$ in $N$ and let $N_\Gamma^\mbb{C}$ be its 
universal complexification. Then the complex homogeneous space 
$N_\Gamma^\mbb{C}/\Gamma$ is Stein, see~\cite{GH} and~\cite[Th\'eor\`eme~1]{L}.

Recall that the complexification $N^\mbb{C}$ is a unipotent complex algebraic 
group and its action on the Siegel domain $\wh{D}$ extends to an algebraic 
action on $\mbb{C}^n$. In particular, the orbits of any algebraic subgroup of 
$N^\mbb{C}$ are closed in $\mbb{C}^n$, see~\cite[Proposition~4.10]{Bo}.

\begin{lem}\label{Lem:free}
The orbit $N_\Gamma\cdot z$ is totally real in $\wh{D}$ if and only if the 
isotropy group $(N_\Gamma^\mbb{C})_z$ is trivial.
\end{lem}

\begin{proof}
For this, let $x+iy\in\lie{n}_\Gamma\oplus i\lie{n}_\Gamma= 
\lie{n}_\Gamma^\mbb{C}$ be a holomorphic vector field on $\wh{D}$ and consider 
$x(z)+j_zy(z)\in T_z\wh{D}$ where $j_z$ is the complex structure of $T_z\wh{D}$, 
for some $z\in \wh{D}$. Suppose that $N_\Gamma\cdot z$ is totally real. Then 
$x(z)+j_zy(z)=0$ implies $x(z)=y(z)=0$, and since $N_\Gamma$ acts freely on 
$\wh{D}$, we obtain $x=y=0$. Consequently, the isotropy $(N_\Gamma^\mbb{C})_z$ 
is discrete. Since $N_\Gamma^\mbb{C}$ acts algebraically and has no finite 
subgroups, we conclude that $(N_\Gamma^\mbb{C})_z$ is trivial.

Conversely, if $N_\Gamma\cdot z$ is not totally real for some $z\in \wh{D}$, 
then there are $x,y\in\lie{n}_\Gamma$ such that $x(z)=j_zy(z)$. Thus the vector 
field $x-iy\in\lie{n}_\Gamma^\mbb{C}$ vanishes at $z$, i.e., $N_\Gamma^\mbb{C}$ 
does not act freely.
\end{proof}

We are now in position to prove the Proposition~\ref{Prop:nec}.

\begin{proof}[Proof of Proposition~\ref{Prop:nec}]
Suppose that $D/\Gamma$ is a Stein manifold. As a first step we are going to 
show that $N_\Gamma$ has at least one totally real orbit in $D$.

Since $D/\Gamma$ is Stein, the domain $D$ admits a $\Gamma$-invariant smooth 
strictly plurisubharmonic function $\rho$ that is exhaustive modulo $\Gamma$. 
We may assume that $0$ is a global minimum of $\rho$. Since $N_\Gamma/\Gamma$ 
is compact, we can suppose without loss of generality that the function $\rho$ 
is invariant under $N_\Gamma$ and is an exhaustion modulo $N_\Gamma$, 
see~\cite[Lemme~2.1]{L}. Let $N_\Gamma\cdot z$ be an orbit lying in the minimal 
set of $\rho$. Then $N_\Gamma\cdot z$ is totally real due to~\cite{HW}.

Finally, due to Lemma~\ref{Lem:free}, it is enough to show that 
$N_\Gamma^\mbb{C}$ acts freely on $N_\Gamma^\mbb{C}\cdot\wh{D}\subset\mbb{C}^n$. 
In order to do this, let $z\in\wh{D}$ and consider the algebraic subgroup 
$H:=(N_\Gamma^\mbb{C})_z$ of $N_\Gamma^\mbb{C}$. Since $N_\Gamma^\mbb{C}\cdot z$ 
is closed in $\mbb{C}^n$, the intersection $\Omega:=\wh{D}\cap( 
N_\Gamma^\mbb{C}\cdot z)$ is Stein and $N_\Gamma$-equivariantly biholomorphic 
to an $N_\Gamma$-invariant Stein open neighborhood $\Omega$ of $N_\Gamma\cdot 
eH\cong N_\Gamma$ in $N_\Gamma^\mbb{C}/H\cong\mbb{C}^k$. As $D/\Gamma$ is Stein 
by assumption, $\Omega/\Gamma$ is also Stein. By applying~\cite{HW} 
and~\cite{L} as above, one sees that $N_\Gamma\cdot eH$ is totally real in 
$N_\Gamma^\mbb{C}/H$. Then, the analog argument as in the proof of 
Lemma~\ref{Lem:free} shows that $H$ is trivial, as wished.
\end{proof}

\begin{rem}
Every orbit lying in the minimal set of $\rho$ is isotropic with respect to the
K\"ahler form $i\partial\ol{\partial}\rho$. In general, it is not isotropic
with respect to the Bergman metric of $D$.
\end{rem}

\subsection{A sufficient condition for $D/\Gamma$ to be Stein}

In this subsection we prove Proposition~\ref{Prop:suff}.

As above, let $\Gamma\subset N$ be a discrete group of unipotent automorphisms 
of a bounded homogeneous domain $D$ and let $N_\Gamma$ be its Zariski closure 
in $N$. Realize $D$ as a Siegel domain of the second kind $\wh{D}\subset 
\mbb{C}^n$ such that $N$ acts by affine transformations on $\wh{D}$.

\begin{proof}[Proof of Proposition~\ref{Prop:suff}]
By assumption, there exists a maximal totally real subalgebra 
$\wh{\lie{n}}_\Gamma\subset\lie{n}$ which contains $\lie{n}_\Gamma$. At the 
group level we thus find a connected subgroup $\wh{N}_\Gamma\subset N$ 
containing $N_\Gamma$ such that $\wh{N}_\Gamma\cdot p_0$ is a maximal totally 
real submanifold of $N\cdot p_0$.

Consider the complexifications $N_\Gamma^\mbb{C}\subset \wh{N}_\Gamma^\mbb{C} 
\subset N^\mbb{C}$ as well as their algebraic actions by affine transformations 
on $\mbb{C}^n$. Since $\wh{N}_\Gamma\cdot p_0$ is maximally totally real, 
the orbit $\wh{N}_\Gamma^\mbb{C}\cdot p_0$ is open in $\mbb{C}^n$ and the 
$\wh{N}_\Gamma^\mbb{C}$-istropy at $p_0$ is trivial, see Lemma~\ref{Lem:free}. 
Since $\wh{N}_\Gamma^\mbb{C}\cdot p_0$ is also closed in $\mbb{C}^n$, we obtain 
$\wh{N}_\Gamma^\mbb{C}\cdot p_0=\mbb{C}^n$. In other words, 
$\wh{N}_\Gamma^\mbb{C}$ acts freely and transitively on $\mbb{C}^n$.

It follows that $N_\Gamma^\mbb{C}$ acts freely and properly on $\mbb{C}^n\cong 
\wh{N}_\Gamma^\mbb{C}$ and that $\mbb{C}^n/N_\Gamma^\mbb{C}\cong\mbb{C}^m$. 
Therefore the corresponding holomorphic principal bundle is holomorphically 
trivial. Together with~\cite[Th\'eor\`eme~1]{L} this implies that 
$\mbb{C}^n/\Gamma\cong(N_\Gamma^\mbb{C}/\Gamma)\times\mbb{C}^m$ is a Stein 
manifold. Since $\wh{D}/\Gamma$ is a locally Stein domain in 
$\mbb{C}^n/\Gamma$, the quotient $\wh{D}/\Gamma$ is likewise Stein 
by~\cite{DocGr}.
\end{proof}

\begin{rem}
The proof of Proposition~\ref{Prop:suff} shows that, if $\lie{n}_\Gamma$ is 
contained in a maximal totally real subalgebra of $\lie{n}$, then 
$N_\Gamma^\mbb{C}$ acts properly and freely on $\mbb{C}^n$. Hence, all 
$N_\Gamma$-orbits are totally real in $D$. Note that this follows also from 
Proposition~\ref{Prop:nec}.
\end{rem}

\begin{rem}
Under the hypotheses of Proposition~\ref{Prop:suff}, $\wh{N}_\Gamma^\mbb{C}\cong
\wh{N}_\Gamma^\mbb{C}\cdot\wh{D}=\mbb{C}^n$ is the universal glo\-balization of 
the induced local $\wh{N}_\Gamma^\mbb{C}$-action on $\wh{D}$ in the sense 
of~\cite{Pal}, see the proof of Proposition~\ref{Prop:suff}.
\end{rem}

\section{The case of the unit ball}\label{Sect:unit}

In this section we consider the unit ball $\mbb{B}_n\subset\mbb{C}^n$. Firstly, 
we illustrate Theorem~\ref{Thm:main} by two examples that can be analyzed by ad 
hoc methods. Then we review the structure of the normal $j$-algebra $\lie{b}_n$ 
of $\mbb{B}_n$, which will be used to show that every totally real subalgebra 
of $\lie{b}_n$ is contained in a maximal totally real one.

\subsection{Two examples}

We identify the unit ball $\mbb{B}_n$ with its unbounded realization
\begin{equation*}
\wh{\mbb{B}}_n=\bigl\{(z,w)\in\mbb{C}\times\mbb{C}^{n-1};\ 
\im(z)-\norm{w}^2>0\bigr\}.
\end{equation*}
For an explicit description of the vector fields belonging to its normal 
$j$-algebra $\lie{b}_n$ as well as of the corresponding one-parameter groups we 
refer the reader to~\cite[Table~1, p.~341]{Mie4}.

First, let us present an example of a nonabelian discrete group $\Gamma$ such 
that $\mbb{B}_3/\Gamma$ is Stein, thus answering the question raised 
in~\cite[Remark~7.6]{Ch} in the negative. Note that $3$ is the smallest
dimension so that a similar example can be constructed.

\begin{ex}
Let us consider the complete holomorphic vector fields 
$x_1:=2iw_1\vf{z}+\vf{w_1}$, $x_2:=2(w_1+w_2)\vf{z}+i\vf{w_1}+i\vf{w_2}$, and 
$x_3:=\vf{z}$ on $\wh{\mbb{B}}_3$. Since their only non-vanishing Lie bracket 
is $[x_1,x_2]=4x_3$, they generate a three-dimensional subalgebra in the 
nilradical of $\lie{b}_3$, isomorphic to the three-dimensional Heisenberg 
algebra. Moreover, this algebra is defined over $\mbb{Q}$ and therefore the 
corresponding connected subgroup admits a cocompact discrete subgroup $\Gamma$, 
which justifies the notation $N_\Gamma=\exp(\mbb{R}x_1\oplus\mbb{R}x_2 
\oplus\mbb{R}x_3)$. One verifies directly that every $N_\Gamma$-orbit in 
$\wh{\mbb{B}}_3$ is maximally totally real. This implies that the universal 
globalization of the local $N_\Gamma^\mbb{C}$-action on $\wh{\mbb{B}}_3$ is 
isomorphic to $N_\Gamma^\mbb{C}$ and the $N_\Gamma$-action on $\mbb{B}_3$ 
corresponds to left multiplication on $N_\Gamma^\mbb{C}$. Since 
$N_\Gamma^\mbb{C}/\Gamma$ is a Stein manifold which contains 
$\wh{\mbb{B}}_3/\Gamma$ as a domain, we see that $\mbb{B}_3/\Gamma$ is Stein 
due to~\cite{DocGr} while $\Gamma$ is not Abelian.
\end{ex}

Secondly, we present an explicit example of a unipotent discrete group $\Gamma$ 
where $\mbb{B}_2/\Gamma$ is non-Stein but admits a (singular) Stein envelope.

\begin{ex}
Let $\Gamma$ be the discrete group consisting of the automorphisms
\begin{equation*}
(z,w)\mapsto\bigl(z+2(n+im)w+i(m^2+n^2)+2k, w+m+in\bigr)=:
\varphi_{2k,m+in}(z,w),
\end{equation*}
where $m,n,k\in\mbb{Z}$ and $(z,w)\in\wh{\mbb{B}}_2$. Although we do not need 
this fact, let us remark that the Lie algebra $\lie{n}_\Gamma$ of the Zariski 
closure of $\Gamma$ coincides with the nilradical of $\lie{b}_2$.

In order to determine the quotient $\wh{\mbb{B}}_2/\Gamma$, we first consider 
the action of the normal subgroup $\Gamma_0:=\{\varphi_{2k,0};\ k\in\mbb{Z}\}
\triangleleft \Gamma$. The map $p\colon\wh{\mbb{B}}_2\to\mbb{C}^*\times 
\mbb{C}$ given by $p(z,w)=(e^{i\pi z},w)$ is $\Gamma_0$-invariant and
yields
\begin{equation*}
\wh{\mbb{B}}_2/\Gamma_0\cong p(\wh{\mbb{B}}_2)=\bigl\{(z,w)\in\mbb{C}^*\times
\mbb{C};\ \abs{z}<e^{-\pi\abs{w}^2}\bigr\}.
\end{equation*}
The induced action of $\mbb{Z}\oplus i\mbb{Z}\cong \Gamma/\Gamma_0$ on 
$p(\wh{\mbb{B}}_2)$ is given by
\begin{equation*}
(m+in)\cdot(z,w)=\bigl(e^{i\pi(2(n+im)w+i(m^2+n^2))}z,w+m+in\bigr).
\end{equation*}
It follows that the action of $\mbb{Z}\oplus i\mbb{Z}$ on $p(\wh{\mbb{B}}_2)$ 
extends to a proper action on the whole of $\mbb{C}^2$ and that the equivariant 
map $\mbb{C}\times\mbb{C}\to\mbb{C}$, $(z,w)\mapsto w$, induces a holomorphic 
line bundle $L:=\mbb{C}^2/(\mbb{Z}\oplus i\mbb{Z})\to E=\mbb{C}/(\mbb{Z}\oplus 
i\mbb{Z})$. We see that $\wh{\mbb{B}}_2/\Gamma\cong p(\wh{\mbb{B}}_2) 
/(\mbb{Z}\oplus i\mbb{Z})$ embeds into $L$ as an open neighborhood of the zero 
section minus this zero section. Note that $\wh{\mbb{B}}_2/\Gamma$ is 
Kobayashi-hyperbolic since $\wh{\mbb{B}}_2$ is so, 
see~\cite[Theorem~3.2.8(2)]{Kob}. Since $\partial\wh{\mbb{B}}_2$ is strictly 
pseudoconvex, the zero-section in $L$ has a strictly pseudoconvex neighborhood 
and hence is negative in the sense of Grauert, see~\cite[Satz~1]{Gr}. It 
follows that the zero section of $L$ can be blown down to yield a Stein space 
$Y$ with an isolated singularity containing $\wh{\mbb{B}}_2/\Gamma$ as a 
neighborhood of this singularity minus the singularity. Consequently, 
$\wh{\mbb{B}}_2/\Gamma$ is holomorphically separable but not Stein.
\end{ex}

\subsection{The normal $j$-algebra of the unit ball}

The automorphism group of the unit ball $\mbb{B}_n$ is $G=\Aut(\mbb{B}_n) 
\cong{\rm{PSU}}(n,1)$. Let $G=KAN$ be an Iwasawa decomposition with maximal 
compact subgroup $K=G_{p_0}\cong{\rm{U}}(n)$ for $p_0=0$. It is well known that 
$G$ is of real rank $1$, i.e., that $\dim A=1$. Moreover, under the 
identification $T_{p_0}\mbb{B}_n\cong \lie{a}\oplus\lie{n}$ we can rewrite the 
moment map condition~\eqref{Eqn:moment} as
\begin{equation}\label{Eqn:bracket}
[x,y](p_0)=\omega_{p_0}\bigl(x(p_0),y(p_0)\bigr)\zeta(p_0)
\end{equation}
for all $x,y\in\lie{n}$, compare~\cite[p.~52]{Pya}. This means that $\lie{n}$ is 
the Heisenberg algebra of dimension $2n-1$ with center $\mbb{R}\zeta$ defined 
by the symplectic form $\omega_{p_0}$ induced by the Bergman metric 
of $\mbb{B}_n$. As we shall see in the following subsection, it is this close 
relation between the geometry of $(\mbb{B}_n,\omega)$ and the structure of $N$ 
that enables us to prove Theorem~\ref{Thm:main} for the unit ball.

Recall from~\cite[Section~4.2]{Mie4} that the normal $j$-algebra 
$\lie{b}_n=\lie{a}\oplus\lie{n}$ can be written as
\begin{equation*}
\lie{a}=\mbb{R}\alpha\quad\text{and}\quad\lie{n}=\bigoplus_{k=1}^{n-1} 
\mbb{R}\xi_k\oplus\bigoplus_{k=1}^{n-1}\mbb{R}\xi'_k\oplus\mbb{R}\zeta,
\end{equation*}
where the only non-zero Lie brackets are
\begin{equation*}
[\xi_k,\xi'_k]=\zeta,[\alpha,\xi_k]=-\xi_k,[\alpha,\xi'_k]=-\xi'_k, 
\text{ and }[\alpha,\zeta]=-2\zeta
\end{equation*}
for all $1\leq k\leq n-1$, and where the complex structure $j\colon\lie{b}_n 
\to\lie{b}_n$ is given by
\begin{equation*}
j(\zeta)=\alpha\text{ and }j(\xi_k)=\xi'_k
\end{equation*}
for all $1\leq k\leq n-1$.

\subsection{Proof of Theorem~\ref{Thm:main} for the unit ball}

Let $\lie{b}_n=\lie{a}\oplus\lie{n}$ be the normal $j$-algebra of $\mbb{B}_n$. 
In this subsection we are going to show that every totally real subalgebra of 
$\lie{n}$ is contained in a maximal totally real subalgebra, which 
generalizes~\cite[Lemma~4.1]{Mie4}. Consequently, we can apply 
Proposition~\ref{Prop:suff} in order to prove Theorem~\ref{Thm:main} for 
$\mbb{B}_n$.

\begin{prop}\label{Prop:completion}
Every totally real subalgebra $\lie{n}'$ of $\lie{n}$ is contained in a maximal 
totally real subalgebra $\wh{\lie{n}}'$ of $\lie{n}$. Moreover, if $\lie{n}'$ 
is Abelian, then $\wh{\lie{n}}'$ can also be chosen to be Abelian.
\end{prop}

\begin{proof} 
Let $\Phi\colon\lie{n}\to\lie{n}\cdot p_0$ be the linear isomorphism given by 
$\Phi(x):=x(p_0)$. Let $V\subset\lie{n}\cdot p_0$ be a real vector subspace. 
The following is a direct consequence of equation~\eqref{Eqn:bracket}.
\begin{enumerate}[(a)]
\item The preimage $\Phi^{-1}(V)$ is an Abelian subalgebra of $\lie{n}$ if and
only if $V$ is isotropic with respect to $\omega_{p_0}$.
\item The preimage $\Phi^{-1}(V)$ is a nonabelian subalgebra of $\lie{n}$ if and
only if $V$ is not isotropic and contains $\zeta(p_0)$.
\end{enumerate}

Now firstly suppose that $\lie{n}'$ is Abelian. Then $\lie{n}'\cdot p_0$ is not 
only totally real but isotropic, and by basic symplectic linear algebra there 
exists a Lagrangian subspace $V$ of $\lie{n}\cdot p_0$ that contains 
$\lie{n}'\cdot p_0$. As we have noted above, the preimage 
$\wh{\lie{n}}':=\Phi^{-1}(V)$ is an Abelian subalgebra of maximal dimension of 
$\lie{n}$ that contains the totally real subalgebra $\lie{n}'$.

In the case that $\lie{n}'$ is not Abelian, we choose any maximally totally real
subspace $V$ of $\lie{n}\cdot p_0$ that contains $\lie{n}'\cdot p_0$. Since
$\lie{n}'\cdot p_0$ contains $\zeta(p_0)$, the same holds for $V$. Hence,
it follows that $\wh{\lie{n}}':=\Phi^{-1}(V)$ has the required properties.
\end{proof}

\section{The case of the Lie ball}\label{Sect:Lie}

In this section we first describe the structure of the normal $j$-algebra 
$\lie{l}_n$ of the Lie ball $\mbb{L}_n$ in order to establish in particular the 
existence of a holomorphic submersion from $\mbb{L}_n$ onto the unit disk 
$\mbb{B}_1$ whose fibers are biholomorphic to $\mbb{B}_{n-1}$. It turns out that 
this submersion is crucial for the proof of Theorem~\ref{Thm:main} since the 
analogon of Proposition~\ref{Prop:completion} does \emph{not} hold true for 
$\lie{l}_n$.

\subsection{The normal $j$-algebra of the Lie ball}

The $n$-dimensional Lie ball is biholomorphically equivalent to the tube domain 
over the symmetric cone
\begin{equation*}
\Omega:=\bigl\{y\in\mbb{R}^n;\ y_n^2-(y_1^2+\dotsb+y_{n-1}^2)>0, y_n>0\bigr\}.
\end{equation*}
For $n=1$ this is just the upper half-plane $\mbb{H}^+$, while for $n=2$ we get 
$\mbb{H}^+\times\mbb{H}^+$. Therefore, we will concentrate on the case $n\geq3$.

In the notation of~\cite[Table~V, Chapter~X]{H} the Lie ball $\mbb{L}_n$ 
corresponds to the item \emph{BD~I} with $p=2$ and $q=n$. The connected 
component of its automorphism group is isomorphic to $G={\rm{SO}}^0(2,n)$ and 
the subgroup of affine automorphisms is $G(\Omega)\ltimes\mbb{R}^n$ with 
$G(\Omega)=\mbb{R}^{>0}{\rm{SO}}^0(1,n-1)$, see~\cite[Chapter~X.5]{FK}.

\begin{rem}
Since the group $G(\Omega)$ acts by matrix multiplication on $\mbb{C}^n$, the 
vector field corresponding to an element $x$ of the Lie algebra of $G(\Omega)$ 
can be obtained by computing $\left.\frac{d}{dt}\right|_0\exp(tx)z$ for 
$z\in\mbb{C}^n$.
\end{rem}

In view of~\cite[Proposition~2.8]{Kan2}, in order to describe the normal 
$j$-algebra $\lie{l}_n=\lie{a}\oplus\lie{n}$ of the Lie ball, it is sufficient 
to find a maximal triangular subalgebra of $\lie{so}(1,n-1)$. By doing so, we 
can choose the basis of $\lie{n}$ given by the vector fields
\begin{align*}
\xi_k&=\vf{z_k}\quad (1\leq k\leq n-2),\\
\xi'_k&=(z_n-z_{n-1})\vf{z_k}+z_k\vf{z_{n-1}}+z_k\vf{z_n}\quad(1\leq k\leq
n-2),\\
\zeta&=\vf{z_{n-1}}+\vf{z_n},\\
\eta&=\vf{z_n}-\vf{z_{n-1}}.
\end{align*}
The only non-vanishing Lie brackets are
\begin{equation*}
[\xi_k,\xi'_k]=\zeta\quad\text{and}\quad[\eta,\xi'_k]=2\xi_k
\end{equation*}
for all $1\leq k\leq n-2$. In particular, we see that $\mbb{R}\zeta$ is the
center of $\lie{n}$ and that the Abelian Lie algebra $\lie{n}'=
\bigoplus_{k=1}^{n-2}\mbb{R}\xi_k\oplus\mbb{R}\zeta$ is its derived
algebra.

A basis of $\lie{a}$ is given by
\begin{align*}
\delta&:=z_1\vf{z_1}+\dotsb+z_n\vf{z_n}\quad\text{(the Euler field), and}\\
\alpha&:=z_n\vf{z_{n-1}}+z_{n-1}\vf{z_n}.
\end{align*}
The action of $\lie{a}$ on $\lie{n}$ is determined by
\begin{equation*}
[\delta,\xi_k]=-\xi_k,\quad[\delta,\zeta]=-\zeta,\quad[\delta,\eta]=-\eta
\end{equation*}
for all $1\leq k\leq n-2$ and by
\begin{equation*}
[\alpha,\xi'_k]=-\xi'_k,\quad[\alpha,\zeta]=-\zeta,\quad[\alpha,\eta]=\eta
\end{equation*}
for all $1\leq k\leq n-2$.

To obtain the complex structure $j\colon\lie{l}_n\to\lie{l}_n$, let us fix the 
base point $p_0=ie_n$ in $\wh{\mbb{L}}_n=\mbb{R}^n+i\Omega$. This yields
\begin{equation*}
j(\xi_k)=\xi'_k,\quad j(\zeta)=\alpha+\delta=:\alpha_1,\quad 
j(\eta)=\delta-\alpha=:\alpha_2
\end{equation*}
for all $1\leq k\leq n-2$.

\subsection{The equivariant fibration of $\mbb{L}_n$}

It follows directly from the bracket relations described above that the 
subalgebra 
\begin{equation*}
\lie{b}_{n-1}:=\mbb{R}\alpha_1\oplus\bigoplus_{k=1}^{n-2}\mbb{R}
\xi_k\oplus \bigoplus_{k=1}^{n-2}\mbb{R}\xi'_k\oplus\mbb{R}\zeta
\end{equation*}
is a $j$-invariant ideal in $\lie{l}_n$ isomorphic to the normal $j$-algebra of 
the $(n-1)$-dimensional unit ball. Moreover, the quotient $\lie{b}_1\cong 
\lie{l}_n/\lie{b}_{n-1}$ is isomorphic to the $j$-subalgebra 
$\mbb{R}\alpha_2\oplus \mbb{R}\eta$ of $\lie{l}_n$. In other words, $\lie{l}_n$ 
decomposes into the semi-direct product $\lie{l}_n=\lie{b}_{n-1} 
\oplus\lie{b}_1$. Let
\begin{equation*}
L_n\cong B_{n-1}\rtimes B_1
\end{equation*}
be the corresponding decomposition at the group level.

Geometrically, this corresponds to the $L_n$-equivariant holomorphic 
submersion $\pi\colon\wh{\mbb{L}}_n\to \mbb{C}$ given by $\pi(z)=z_n-z_{n-1}$. 
In fact, an elementary argument shows that for $z\in\wh{\mbb{L}}_n$ we have 
$z_n-z_{n-1}\in\mbb{H}^+$ and $z_n+z_{n-1}\in\mbb{H}^+$. Moreover, for 
$a:=z_n-z_{n-1}\in\mbb{H}^+$ and $w:=z_n+z_{n-1}\in\mbb{H}^+$, we obtain
\begin{equation*}
\pi^{-1}(a)\cong\left\{\left.(z_1,\dotsc,z_{n-2},w)\in\mbb{C}^{n-1}\right|\ 
\im(w)>\frac{\im(z_1)^2+\dotsb+\im(z_{n-2})^2}{\im(a)}\right\},
\end{equation*}
which is a realization of the $(n-1)$-dimensional unit ball as an unbounded
tube domain. Hence, all $\pi$-fibers are isomorphic to $\mbb{B}_{n-1}$.

\begin{rem}
Since the Lie ball $\mbb{L}_n$ is irreducible for $n\geq3$, the holomorphic 
submersion $\pi$ is \emph{not} holomorphically locally trivial unless $n=2$,
as follows from~\cite{R}.
\end{rem}

\subsection{Proof of Theorem~\ref{Thm:main} for the Lie ball}

Let $\Gamma$ be a discrete subgroup of $\Aut^0(\mbb{L}_n)$ such that 
$\Gamma\subset N$. Let $N_\Gamma$ be the Zariski closure of $\Gamma$ in $N$ and 
suppose that $\lie{n}_\Gamma$ is totally real, i.e., that $\lie{n}_\Gamma\cap 
j(\lie{n}_\Gamma)=\{0\}$.

We start by giving an example that shows that Proposition~\ref{Prop:completion} 
does no longer hold true for the Lie ball.

\begin{ex}
Let us consider the $6$-dimensional normal $j$-algebra 
$\lie{l}_3=\lie{a}\oplus\lie{n}$ of the $3$-dimensional Lie ball. We claim that 
the totally real subalgebra $\mbb{R}\xi'_1\subset\lie{n}$ is not contained in a 
maximal totally real subalgebra of $\lie{n}$.

Since $\lie{b}_2\triangleleft\lie{l}_3$ is $4$-dimensional and not totally real, 
a maximal totally real subalgebra of $\lie{n}$ must contain an element 
$x=a\xi_1+b\xi_1'+c\zeta +d\eta$ with $a,b,c,d\in\mbb{R}$ and $d\not=0$. Thus it 
also must contain
\begin{equation*}
[x,\xi_1']=a\zeta+2d\xi_1\quad\text{and}\quad\bigl[[x,\xi_1'],\xi_1'\bigr] = 
2d\zeta.
\end{equation*}
Therefore it contains $\zeta$ and $\xi_1=-j(\xi_1')$ as well, which is 
impossible.
\end{ex}

The following proposition is weaker than Proposition~\ref{Prop:completion} but 
will still allow us to prove Theorem~\ref{Thm:main} for the Lie ball. For its 
statement and proof we decompose the nilradical $\lie{n}$ of $\lie{l}_n$ as
$\lie{n}=\lie{n}_{n-1}\oplus\mbb{R}\eta$ where the ideal $\lie{n}_{n-1}:= 
\lie{b}_{n-1}\cap\lie{n}$ is isomorphic to the nilradical of $\lie{b}_{n-1}$.

\begin{prop}\label{Prop:completion2}
Let $\lie{n}'$ be a totally real subalgebra of the nilradical $\lie{n}= 
\lie{n}_{n-1}\oplus\mbb{R}\eta$. If $\lie{n}'$ is not contained in 
$\lie{n}_{n-1}$, then there exists a maximal totally real subalgebra 
$\wh{\lie{n}}'$ of $\lie{n}$ which contains $\lie{n}'$.
\end{prop}

\begin{proof}
Since by assumption $\lie{n}'$ is not contained in $\lie{n}_{n-1}$, there is an 
element of the form
\begin{equation*}
x_0=\sum a_k\xi_k + \sum b_l\xi'_l + c\zeta + \eta
\end{equation*}
which belongs to $\lie{n}'$. Moreover, we have $\dim(\lie{n}'\cap\lie{n}_{n-1}) 
=\dim\lie{n}'-1$.

As a first step, we claim that $\lie{n}'$ is totally real if and only if 
$\lie{n}'\cap\lie{n}_{n-1}$ is so. In order to see this, suppose that 
$\lie{n}'\cap\lie{n}_{n-1}$ is totally real, the other implication being 
trivial. Choose an element $x\in \lie{n}'\cap j(\lie{n}')$ and decompose it as 
$x=x_1+tx_0$ where $x_1\in\lie{n}'\cap\lie{n}_{n-1}$ and $t\in\mbb{R}$. By 
assumption,  we have $j(x)=j(x_1)+tj(x_0)\in\lie{n}'$. Then, from 
$j(\eta)=\alpha_2$ we deduce $t=0$, hence $x=x_1\in\lie{n}'\cap\lie{n}_{n-1}$. 
Since the latter algebra is assumed to be totally real we obtain $x=0$, as was 
to be shown.

Next, observe that if $\zeta\notin\lie{n}'$, then $\lie{n}'\oplus\mbb{R}\zeta$ 
is again a totally real subalgebra, for the following reason. If 
$\zeta\notin\lie{n}'$, then $\lie{n}'\cap \lie{n}_{n-1}$ is Abelian. Therefore, 
the Lie algebra $(\lie{n}'\cap\lie{n}_{n-1})\oplus\mbb{R}\zeta$ is also 
Abelian, hence totally real, which due to the previous step implies that 
$\lie{n}'\oplus\mbb{R}\zeta$ is totally real.

Consequently, we may assume that $\zeta\in\lie{n}'$. Let us consider the 
projection $\pi\colon\lie{n}_{n-1}\to\lie{n}_{n-1}$ onto $\bigoplus_k\mbb{R} 
\xi_k\oplus\mbb{R}\zeta$ with kernel $\bigoplus_l\mbb{R}\xi'_l$. If the 
restriction
\begin{equation*}
\pi|_{\lie{n}'\cap\lie{n}_{n-1}}\colon\lie{n}'\cap\lie{n}_{n-1}\to 
\bigoplus_k\mbb{R}\xi_k\oplus\mbb{R}\zeta
\end{equation*}
is surjective, then $\dim(\lie{n}'\cap\lie{n}_{n-1})\geq n-1$, hence $\dim 
\lie{n}'=n$, and we are done.

Therefore, suppose that $\xi_{k_0}\notin\pi(\lie{n}'\cap\lie{n}_{n-1})$ for 
some $k_0$. Note that in particular $\xi_{k_0}\notin\lie{n}'$. From 
$\zeta\in\lie{n}'$ it follows that $(\lie{n}'\cap\lie{n}_{n-1}) 
\oplus\mbb{R}\xi_{k_0}$ is a subalgebra of $\lie{n}'\oplus\lie{n}_{n-1}$. 
Since moreover, $\eta$ and $\xi_{k_0}$ commute, we see that $\lie{n}'\oplus 
\mbb{R}\xi_{k_0}$ is a subalgebra of $\lie{n}$. We claim that 
$\lie{n}'\oplus\mbb{R}\xi_{k_0}$ is totally real. In order to see this it is 
sufficient to show that $(\lie{n}'\cap\lie{n}_{n-1})\oplus\mbb{R}\xi_{k_0}$ is 
totally real. This, however, follows from the fact that $\bigl[x_0, 
j(\xi_{k_0})\bigr]$ has a non-zero contribution from $\xi_{k_0}$.

Iterating this procedure we eventually construct a maximal totally real 
subalgebra of $\lie{n}$ that contains $\lie{n}'$, as was to be shown.
\end{proof}

\begin{proof}[Proof of Theorem~\ref{Thm:main} for the Lie ball]
We shall prove that $\mbb{L}_n/\Gamma$ is Stein where $\Gamma$ satisfies the 
conditions stated at the beginning of this subsection.

Let us first assume that $\lie{n}_\Gamma$ is not contained in the $j$-ideal 
$\lie{b}_{n-1}$. Due to Proposition~\ref{Prop:completion2} there exists a 
maximal totally real subalgebra $\wh{\lie{n}}_\Gamma\subset\lie{n}$ that 
contains $\lie{n}_\Gamma$. Then the result follows from 
Proposition~\ref{Prop:suff}.

If $\lie{n}_\Gamma$ is contained in $\lie{b}_{n-1}$, we find a maximal totally 
real subalgebra $\wh{\lie{n}}_\Gamma$ of $\lie{b}_{n-1}$ that contains 
$\lie{n}_\Gamma$, see Proposition~\ref{Prop:completion}. Let $\wh{N}_\Gamma$ be 
the corresponding closed subgroup of $\Aut^0(\mbb{L}_n)$. The holomorphic 
submersion $\pi\colon\wh{\mbb{L}}_n\to\mbb{H}^+$ is $\wh{N}_\Gamma$-invariant 
and extends to a $\wh{N}_\Gamma^\mbb{C}$-invariant holomorphic submersion 
$\wh{\pi}\colon\wh{N}_\Gamma^\mbb{C}\cdot\wh{\mbb{L}}_n\to\mbb{H}^+$. Then, 
$\wh{N}_\Gamma^\mbb{C}$ acts freely and properly on 
$\wh{N}_\Gamma^\mbb{C}\cdot\wh{\mbb{L}}_n$, and the fibers of $\wh{\pi}$ are 
precisely the $\wh{N}_\Gamma^\mbb{C}$-orbits, isomorphic to $\mbb{C}^{n-1}$. In 
other words, $\wh{\pi}$ defines an $\wh{N}_\Gamma^\mbb{C}$-principal bundle. 
Since $\mbb{H}^+$ is contractible, this $\wh{N}_\Gamma^\mbb{C}$-principal 
bundle is holomorphically trivial, i.e., we have
\begin{equation*}
\bigl(\wh{N}_\Gamma^\mbb{C}\cdot\wh{\mbb{L}}_n\bigr)/\Gamma 
\cong(\wh{N}_\Gamma^\mbb{C}/\Gamma)\times\mbb{H}^+.
\end{equation*}
Since $\wh{N}_\Gamma^\mbb{C}/\Gamma$ is Stein due to~\cite[Th\'eor\`eme~1]{L}, 
the same is true for $\bigl(\wh{N}_\Gamma^\mbb{C}\cdot\wh{\mbb{L}}_n\bigr)/ 
\Gamma$. Hence, the theorem of Docquier-Grauert~\cite{DocGr} implies that 
$\wh{\mbb{L}}_n/\Gamma$ is Stein as well.
\end{proof}

\section{The case of the Siegel disk}\label{Sect:Siegel}

In this section we present an example that shows that the analoga of 
Theorem~\ref{Thm:main} and Proposition~\ref{Prop:completion2} do not hold true 
for arbitrary bounded homogeneous domains. This example will be constructed on a 
$5$-dimensional bounded homogeneous domain holomorphically embedded in the 
$6$-dimensional Siegel disk $\mbb{S}_3$. Therefore we describe as a first step 
the normal $j$-algebra $\lie{s}_n$ of $\mbb{S}_n$.

\subsection{The normal $j$-algebra of Siegel's upper half-plane}

Recall that the Siegel disk $\mbb{S}_n$ can be realized as Siegel's upper half 
plane which is the symmetric tube domain associated with the cone of positive 
definite real symmetric matrices. The automorphism group of $\mbb{S}_n$ is 
isomorphic to the symplectic group ${\rm{Sp}}(n,\mbb{R})$ and $\mbb{S}_n$ 
corresponds to the hermitian symmetric space of type \emph{C~I}  in~\cite{H}. 
For $n=1$, Siegel's upper half-plane is the usual upper half-plane, while, for 
$n=2$, it is isomorphic to the $3$-dimensional Lie ball.

The linear automorphism group of the symmetric cone of positive definite 
real symmetric matrices is ${\rm{GL}}(n,\mbb{R})$ acting by $g\cdot A:= gAg^t$,
see~\cite[p.~213]{FK}. Hence, the normal $j$-algebra $\lie{s}_n$ of $\mbb{S}_n$ 
can be determined as before using~\cite[Proposition~X.5.4]{FK} 
and~\cite[Proposition~2.8]{Kan2}. Our description of $\lie{s}_n$ 
follows~\cite{Is}.

Let us denote by $\lie{u}_n$ the solvable Lie algebra of lower triangular real 
$n\times n$ matrices. We have the decomposition $\lie{u}_n=\lie{a}_n 
\oplus\lie{u}'_n$ where $\lie{a}_n$ is the Abelian Lie algebra of diagonal 
matrices in $\mbb{R}^{n\times n}$. The normal $j$-algebra of Siegel's upper half 
plane can be realized as
\begin{equation*}
\lie{s}_n:=\left\{\left.
\begin{pmatrix}
A&B\\0&-A^t
\end{pmatrix}
\right|\ A\in\lie{u}_n,\ B\in\Sym(n,\mbb{R})\right\}.
\end{equation*}
Note that $\varphi\colon\lie{u}_n\to\Sym(n,\mbb{R})$, $\varphi(A)=A+A^t$, is a
linear isomorphism. The complex structure $j\colon\lie{s}_n\to\lie{s}_n$ is  
given by
\begin{equation*}
j
\begin{pmatrix}
A&B\\0&-A^t
\end{pmatrix}=
\begin{pmatrix}
\varphi^{-1}(B)&-\varphi(A)\\0&-\varphi^{-1}(B)^t
\end{pmatrix}.
\end{equation*}
The linear form $\lambda\in\lie{s}_n^*$ corresponding to the Bergman metric is 
given by
\begin{equation*}
\lambda
\begin{pmatrix}
A&B\\0&-A^t
\end{pmatrix}=\tr(B).
\end{equation*}

The elements of the solvable Lie group $S_n$ are of the form $\left( 
\begin{smallmatrix}A&B\\0&(A^t)^{-1}\end{smallmatrix}\right)$ where $A$ is 
lower triangular and $B$ is symmetric. An element of the form 
$\left(\begin{smallmatrix}A&0\\0&(A^t)^{-1}\end{smallmatrix}\right)$ acts on 
$\mbb{S}_n$ by $Z\mapsto AZA^t$ while the elements of the form $\left( 
\begin{smallmatrix}I_n&B\\0&I_n\end{smallmatrix}\right)$ act by translation.
The vector field induced by $x\in\lie{s}_n$ is given by $\left.\frac{d}{dt} 
\right|_0\exp(tx)\cdot Z$ for $Z\in\mbb{S}_n$.

One verifies directly that the subspace of $\lie{s}_n$ consisting of all
matrices $A=(a_{kl})$ and $B=(b_{kl})$ such that $a_{kl}=b_{kl}=0$ for all 
$1\leq k,l\leq n-1$ is a $j$-invariant ideal isomorphic to $\lie{b}_n$, while 
the subspace consisting of all matrices where $a_{nl}=b_{nl}=0$ for all $l$ is 
a $j$-invariant complementary subalgebra isomorphic to $\lie{s}_{n-1}$. 
Consequently, $\lie{s}_n$ decomposes as a semi-direct sum $\lie{s}_{n-1} 
\oplus\lie{b}_n$. Geometrically, this decomposition corresponds to the 
equivariant fibration $\pi_n\colon\wh{\mbb{S}}_n\to\wh{\mbb{S}}_{n-1}$ where 
$\pi_n(Z)$ is the submatrix of $Z$ consisting of the first $n-1$ lines and 
columns.

In the following we will concentrate on the case $n=3$. Here we have
$\lie{s}_3=\lie{b}_3\oplus\lie{b}_2\oplus\lie{b}_1$ where $\lie{b}_3$ is a 
$j$-ideal in $\lie{s}_3$ and $\lie{b}_2$ is a $j$-ideal in 
$\lie{s}_2=\lie{b}_2\oplus\lie{b}_1$. The composition of the two equivariant 
fibrations is the map $\pi:=\pi_2\circ\pi_3\colon\wh{\mbb{S}}_3\to\mbb{H}^+$ 
given by $\pi(Z)=z_{11}$. For the rest of this section let $D:=\pi^{-1}(i)$ be 
the $5$-dimensional bounded homogeneous (non-symmetric) domain corresponding to 
the normal $j$-algebra $\lie{b}_3\oplus\lie{b}_2$. Let us consider the bases 
$(\alpha_3,\xi_{31},\xi_{32},\xi'_{31},\xi'_{32},\zeta_3)$ of $\lie{b}_3$ and 
$(\alpha_2,\xi_{21},\xi'_{21},\zeta_2)$ of $\lie{b}_2$. Let us realize these 
elements explicitly as matrices as well as vector fields in the coordinates of 
the entries of $Z=(z_{kl})\in\Sym(3,\mbb{C})$:
\begin{align*}
\zeta_3&=-2
\begin{pmatrix}
0&E_{33}\\0&0
\end{pmatrix}\mapsto-2\vf{z_{33}}, &
\alpha_3&=
\begin{pmatrix}
E_{33}&0\\0&-E_{33}
\end{pmatrix}\mapsto 
z_{13}\vf{z_{13}}+z_{23}\vf{z_{23}}+2z_{33}\vf{z_{33}}, \\
\xi_{31}&=
\begin{pmatrix}
0&E_{13}+E_{31}\\0&0
\end{pmatrix}\mapsto\vf{z_{13}}, &
\xi'_{31}&=
\begin{pmatrix}
E_{31}&0\\0&-E_{13}
\end{pmatrix}\mapsto z_{11}\vf{z_{13}}+z_{12}\vf{z_{23}}+2z_{13}\vf{z_{33}},\\
\xi_{32}&=
\begin{pmatrix}
0&E_{23}+E_{32}\\0&0
\end{pmatrix}\mapsto\vf{z_{23}}, &
\xi'_{32}&=
\begin{pmatrix}
E_{32}&0\\0&-E_{23}
\end{pmatrix}\mapsto z_{12}\vf{z_{13}}+z_{22}\vf{z_{23}}+2z_{23}\vf{z_{33}},\\
\zeta_2&=-2
\begin{pmatrix}
0&E_{22}\\0&0
\end{pmatrix}\mapsto-2\vf{z_{22}}, &
\alpha_2&=
\begin{pmatrix}
E_{22}&0\\0&-E_{22}
\end{pmatrix}\mapsto 
z_{12}\vf{z_{12}}+2z_{22}\vf{z_{22}}+z_{23}\vf{z_{23}},\\
\xi_{21}&=
\begin{pmatrix}
0&E_{12}+E_{21}\\0&0
\end{pmatrix}\mapsto\vf{z_{12}}, &
\xi'_{21}&=
\begin{pmatrix}
E_{21}&0\\0&-E_{12}
\end{pmatrix}\mapsto z_{11}\vf{z_{12}}+2z_{12}\vf{z_{22}}+z_{13}\vf{z_{23}}.
\end{align*}

For instance, in order to find the vector field belonging to $\xi'_{31}$ we 
calculate
\begin{equation*}
\exp(t\xi'_{31})\cdot Z=
\begin{pmatrix}
1&0&0\\0&1&0\\t&0&1
\end{pmatrix}Z
\begin{pmatrix}
1&0&t\\0&1&0\\0&0&1
\end{pmatrix}=
\begin{pmatrix}
z_{11} & z_{12} & z_{13}+tz_{11}\\
z_{12} & z_{22} & z_{23}+tz_{12}\\
z_{13}+tz_{11} & z_{23}+tz_{12} & z_{33}+2tz_{13}+t^2z_{11}
\end{pmatrix}
\end{equation*}
and then derive with respect to $t$.

The representation of $\lie{b}_2$ on $\lie{b}_3$ is defined by
\begin{align*}
[\alpha_2,\xi_{32}]&=\xi_{32} &
[\alpha_2,\xi'_{32}]&=-\xi'_{32}\\
[\xi_{21},\xi'_{31}]&=-\xi_{32} &
[\xi_{21},\xi'_{32}]&=-\xi_{31}\\
[\xi'_{21},\xi_{31}]&=\xi_{32} &
[\xi'_{21},\xi'_{32}]&=-\xi'_{31}\\
[\zeta_2,\xi'_{32}]&=2\xi_{32},
\end{align*}
all other brackets being zero.

\subsection{A counterexample}

We present an example that shows that the analoga of 
Proposition~\ref{Prop:completion2} and Theorem~\ref{Thm:main} do not hold true 
for $D$.

Let $\lie{n}$ denote the nilradical of $\lie{b}_3\oplus\lie{b}_2$ and consider 
the elements $x_1,x_2,x_3\in\lie{n}$ given by
\begin{equation*}
x_1:=\xi'_{31}+\zeta_3+\xi_{21},\ x_2:=-\xi_{31}+\xi'_{21},\ x_3:=\zeta_3+ 
\zeta_2.
\end{equation*}
Since $[x_1,x_2]=x_3$, they generate a Lie subalgebra $\lie{n}_\Gamma:=  
\mbb{R}x_1\oplus\mbb{R}x_2\oplus\mbb{R}x_3$ that is isomorphic to the 
$3$-dimensional Heisenberg algebra and projects surjectively onto the nilradical 
of $\lie{b}_2$. Moreover, the corresponding group $N_\Gamma$ of automorphisms 
of $D$ admits a cocompact discrete subgroup, which justifies the notation 
$\lie{n}_\Gamma$.

We shall see that all $N_\Gamma$-orbits in $D$ are totally real, while 
$\lie{n}_\Gamma$ is not contained in any maximal totally real subalgebra of 
$\lie{n}$. This shows that the analogon of Proposition~\ref{Prop:completion2} 
does not hold true for the domain $D$, even under the stronger assumption that 
\emph{all} $N_\Gamma$-orbits are totally real in $D$.

\begin{rem}
A subalgebra of codimension $1$ in a nilpotent Lie algebra is automatically 
an ideal. We apply this result in the following way: If $y\in\lie{n}$ is an 
element such that $\lie{n}_\Gamma\oplus\mbb{R}y$ is a subalgebra, then $y$ 
normalizes $\lie{n}_\Gamma$.
\end{rem}

\begin{lem}
The Lie algebra $\lie{n}_\Gamma$ is not contained in a maximal totally real 
subalgebra of $\lie{n}$. Hence, the analogon of 
Proposition~\ref{Prop:completion2} does not hold true for the domain $D$.
\end{lem}

\begin{proof}
Suppose for a moment that $\wh{\lie{n}}_\Gamma$ is a maximal totally real 
subalgebra of $\lie{n}$ with $\lie{n}_\Gamma\subset\wh{\lie{n}}_\Gamma$ and let 
$y\in\wh{\lie{n}}_\Gamma\setminus\lie{n}_\Gamma$. Writing
\begin{equation*}
y=a_1\xi_{31}+a_2\xi_{32}+a_1'\xi_{31}'+a_2'\xi_{32}'+c_3\zeta_3+b_1\xi_{21}+ 
b_1'\xi_{21}'+c_2\zeta_2
\end{equation*}
we find
\begin{equation*}
[y,x_1]=a_2'\xi_{31}+(a_1'-b_1)\xi_{32}+a_1\zeta_3-b_1'\zeta_2\quad\text{and}
\quad\bigl[[y,x_1],x_1\bigr]=a_2'\zeta_3.
\end{equation*}
If $a_2'\not=0$, then $\wh{\lie{n}}_\Gamma$ would contain $\zeta_3$ and 
$\zeta_2$ and consequently $\xi_{31}'+\xi_{21}=-j(x_2)$ as well, thus 
contradicting the fact that $\wh{\lie{n}}_\Gamma$ is totally real. Hence, we 
must have $a_2'=0$.

It follows that the elements
\begin{align*}
[y,x_1]&=(a_1'-b_1)\xi_{32}+a_1\zeta_3-b_1'\zeta_2\quad\text{and}\\
[y,x_2]&=-(a_1+b_1')\xi_{32}+a_1'\zeta_3+b_1\zeta_2
\end{align*}
belong to $\wh{\lie{n}}_\Gamma$. Consequently, $\wh{\lie{n}}_\Gamma$ contains 
also the element
\begin{equation*}
(a_1+b_1')[y,x_1]+(a_1'-b_1)[y,x_2]=
(a_1^2-a_1'b_1+a_1b_1'+a_1'^2)\zeta_3-(b_1^2-a_1'b_1+a_1b_1'+b_1'^2)\zeta_2.
\end{equation*}
As above, since $\wh{\lie{n}}_\Gamma$ is totally real, neither $\zeta_3$ nor 
$\zeta_2$ belong to $\wh{\lie{n}}_\Gamma$. This gives
\begin{equation*}
a_1'^2-a_1'b_1+a_1b_1'+a_1^2=-b_1^2+a_1'b_1-a_1b_1'-b_1'^2,
\end{equation*}
which is equivalent to
\begin{equation*}
(a_1+b_1')^2+(a_1'-b_1)^2=0.
\end{equation*}

Hence, we obtain $a_1=-b_1'$ and $a_1'=b_1$. In summary, this yields
\begin{equation*}
y=a_1(\xi_{31}-\xi_{21}')+a_2\xi_{32}+a_1'(\xi_{31}'+\xi_{21})+c_3\zeta_3+ 
c_2\zeta_2.
\end{equation*}
In particular $y$ normalizes $\lie{n}_\Gamma$.

Adding $a_1x_2-a_1'x_1-c_2x_3$ to $y$, we conclude that $\wh{\lie{n}}_\Gamma 
\setminus\lie{n}_\Gamma$ contains an element of the form 
\begin{equation}\label{Eqn:ytau}
y_\tau:=\xi_{32}+\tau\zeta_3
\end{equation}
for some $\tau\in\mbb{R}$. Observe that $y_\tau$ centralizes $\lie{n}_\Gamma$ 
and that $\lie{n}_\Gamma\oplus\mbb{R}y_\tau$ is a totally real subalgebra of 
$\lie{n}$.

According to the above remark, $\lie{n}_\Gamma\oplus\mbb{R}y_\tau$ must be a 
normal subalgebra of $\wh{\lie{n}}_\Gamma$. Therefore, we calculate its 
normalizer and find
\begin{equation*}
\mathcal{N}_{\lie{n}}(\lie{n}_\Gamma\oplus\mbb{R}y_\tau)
=\lie{n}_\Gamma\oplus\mbb{R}\xi_{32}\oplus\mbb{R}\zeta_3 
=\mbb{R}(\xi'_{31}+\xi_{21})\oplus\mbb{R}(-\xi_{31}+
\xi'_{21})\oplus\mbb{R}\xi_{32}\oplus\mbb{R}\zeta_3\oplus\mbb{R}\zeta_2.
\end{equation*}
Since the normalizer is $5$-dimensional, it must coincide with 
$\wh{\lie{n}}_\Gamma$. This however contradicts our assumption because the 
normalizer is \emph{not} totally real.
\end{proof}

Let $\wh{N}_\Gamma$ be the connected Lie group having Lie algebra 
$\lie{n}_\Gamma\oplus\mbb{R}y_\tau$ where $y_\tau$ is defined in 
Equation~\eqref{Eqn:ytau}. Since $y_\tau$ centralizes $\lie{n}_\Gamma$, the 
group $\wh{N}_\Gamma$ admits cocompact discrete subgroups. We will show that the 
quotient of $D$ with respect to any cocompact discrete subgroup of 
$\wh{N}_\Gamma$ is \emph{not} Stein, hence that the analogon of 
Theorem~\ref{Thm:main} does not hold for $D$. For this it is enough to find an 
$\wh{N}_\Gamma$-orbit in $D$ which is not totally real. Note that this implies 
again that $\lie{n}_\Gamma\oplus\mbb{R}y_\tau$ cannot be contained in a maximal 
totally real subalgebra. Indeed, for every $\alpha>(1+\tau^2)^{-2}$ the matrix
\begin{equation*}
Z_0=
\begin{pmatrix}
i & \frac{\tau+i}{1+\tau^2} & 0\\
\frac{\tau+i}{1+\tau^2} & i\alpha & 0\\
0 & 0 & i
\end{pmatrix}
\end{equation*}
belongs to $D$, and one sees without difficulty that $\wh{N}_\Gamma\cdot Z_0$ 
is not totally real.

Nevertheless, we shall prove in the rest of this subsection that $D/\Gamma$ is 
a Stein manifold where $\Gamma$ is any cocompact discrete subgroup of 
$N_\Gamma$. Using the description of the elements of $\lie{s}_n$ as matrices, we 
can realize $N_\Gamma^\mbb{C}$ as the matrix group
\begin{equation}\label{Eqn:MatrixGroup}
N_\Gamma^\mbb{C}=\left\{\left.
\begin{pmatrix}
1&0&0&0&a&-b\\
b&1&0&a&-2c&-\frac{a^2+b^2}{2}\\
a&0&1&-b&\frac{a^2+b^2}{2}&-2(a+c)\\
0&0&0&1&-b&-a\\
0&0&0&0&1&0\\
0&0&0&0&0&1\\
\end{pmatrix}\right|\ a,b,c\in\mbb{C}\right\}.
\end{equation}
The action of $N_\Gamma^\mbb{C}$ on $\Sym(3,\mbb{C})$ is given by
\begin{multline*}
Z=(z_{kl})\mapsto\\
\begin{pmatrix}
z_{11} & z_{12}+bz_{11}+a & z_{13}+az_{11}-b\\
z_{12}+bz_{11}+a & z_{22}+2bz_{12}+b^2z_{11}+ab-2c & 
z_{23}+az_{12}+bz_{13}+abz_{11}+\frac{a^2-b^2}{2}\\
z_{13}+az_{11}-b & z_{23}+az_{12}+bz_{13}+abz_{11}+\frac{a^2-b^2}{2} & 
z_{33}+2az_{13}+a^2z_{11}-ab-2(a+c)
\end{pmatrix}.
\end{multline*}
It is not difficult to verify that $N_\Gamma^\mbb{C}$ acts freely on 
$\Sym(3,\mbb{C})$, which in particular implies that $N_\Gamma$ has only totally 
real orbits in $D$, see Lemma~\ref{Lem:free}.

In the following we will show that $N_\Gamma^\mbb{C}$ acts properly on 
$\Sym(3,\mbb{C})$ with quotient manifold $\Sym(3,\mbb{C})/N_\Gamma^\mbb{C} 
\cong\mbb{C}^3$. It then follows that the principal bundle $\Sym(3,\mbb{C})\to 
\mbb{C}^3$ is holomorphically trivial. Consequently, $\Sym(3,\mbb{C})/\Gamma$ 
and $D/\Gamma$ are Stein.

In order to prove that the $N_\Gamma^\mbb{C}$-action is proper we shall use the 
following lemma.

\begin{lem}\label{Lem:PropernessinSteps}
Let $G$ be a Lie group acting smoothly and freely on a manifold $M$ and let $H$ 
be a closed normal subgroup of $G$. Then $G$ acts properly on $M$ if and only if
$H$ acts properly on $M$ and $G/H$ acts properly on $M/H$.
\end{lem}

\begin{proof}
Firstly, suppose that $G$ acts properly on $M$. Since $H$ is closed in $G$, the 
$H$-action on $M$ is proper and we get a smooth action of $G/H$ on the quotient 
manifold $M/H$. Properness of this latter action was shown 
in~\cite[Proposition~1.3.2]{P}.

Conversely, suppose that the actions of $H$ on $M$ and of $G/H$ on $M/H$ are 
proper. Let $(g_n)$ and $(p_n)$ be sequences in $G$ and $M$, respectively, such 
that $(g_n\cdot p_n,p_n)$ converges to $(q_0,p_0)$. We must show that $(g_n)$ 
has a convergent subsequence.

For this, note that $\bigl(g_nH\cdot[p_n],[p_n]\bigr)$ converges to 
$\bigl([q_0],[p_0]\bigr)$ in $M/H\times M/H$ where $[p]\in M/H$ denotes the
class of $p$ modulo $H$. Thus $(g_nH)$ has a convergent subsequence. Without 
loss of generality we assume that $g_nH\to g_0H$. Hence, there is a sequence 
$(h_n)$ in $H$ such that $h_n^{-1}g_n\to g_0$ in $G$, from which we conclude 
that
\begin{equation*}
\bigl( h_n\cdot(h_n^{-1}g_n\cdot p_n), h_n^{-1}g_n\cdot p_n\bigr)\to (q_0, 
g_0\cdot p_0).
\end{equation*}
Since $H$ acts properly on $M$, it follows that $(h_n)$ and hence $(g_n)$ have 
convergent subsequences.
\end{proof}

We first apply Lemma~\ref{Lem:PropernessinSteps} to the action of the center of 
$N_\Gamma^\mbb{C}$ on $\Sym(3,\mbb{C})$. The center of $N_\Gamma^\mbb{C}$ is 
given by all matrices of the form~\eqref{Eqn:MatrixGroup} having $a=b=0$ and 
acts on $\Sym(3,\mbb{C})$ by
\begin{equation*}
c\cdot\begin{pmatrix}
z_{11} & z_{12} & z_{13}\\
z_{12} & z_{22} & z_{23}\\
z_{13} & z_{23} & z_{33}
\end{pmatrix}=
\begin{pmatrix}
z_{11} & z_{12} & z_{13}\\
z_{12} & z_{22}-2c & z_{23}\\
z_{13} & z_{23} & z_{33}-2c
\end{pmatrix}.
\end{equation*}
Hence, the corresponding $\mbb{C}$-principal bundle is holomorphically trivial 
and given by
\begin{equation*}
\pi\colon\Sym(3,\mbb{C})\to\mbb{C}^5,\quad
\pi(z_{kl})=(z_{11},z_{12},z_{13},z_{22}-z_{33},z_{23}).
\end{equation*}
The induced action of $\mbb{C}^2\cong N_\Gamma^\mbb{C}/ Z(N_\Gamma^\mbb{C})$ on 
$\mbb{C}^5$ is given by
\begin{equation*}
(a,b)\cdot 
z=
\begin{pmatrix}
z_1\\
z_2+bz_1+a\\
z_3+az_1-b\\
z_4+2bz_2-2az_3+(b^2-a^2)z_1+2ab+2a\\
z_5+az_2+bz_3+abz_1+(a^2-b^2)/2
\end{pmatrix}.
\end{equation*}
In the next step we consider only the action of the one parameter group 
$\mbb{C}_a$ of elements $(a,0)$. A direct calculation shows that the 
$\mbb{C}_a$-action on $\mbb{C}^5$ becomes a translation in the first coordinate 
after conjugation by the biregular map $\Phi\colon\mbb{C}^5\to\mbb{C}\times 
\mbb{C}^4$,
\begin{equation*}
z\mapsto(z_2,z_1,z_3-z_1z_2,z_4+2z_2z_3-2z_1z_5-2z_2,z_2^2-2z_5).
\end{equation*}

It follows that $\mbb{C}_a$ acts properly on $\mbb{C}^5$ and that the 
corresponding $\mbb{C}_a$-principal bundle $\mbb{C}^5\to 
\mbb{C}^5/\mbb{C}_a\cong \mbb{C}^4$ is holomorphically trivial. Moreover, the 
induced $\mbb{C}_b$-action on $\mbb{C}^4$ is of the form
\begin{equation*}
b\cdot w=
\begin{pmatrix}
w_1\\
w_2-b(w_1^2+1)\\
w_3-2bw_1\\
w_4-2bw_2+b^2(w_1^2+1)\\
\end{pmatrix}.
\end{equation*}

Note that the projection onto the first three coordinates is equivariant; 
therefore we obtain a free algebraic $\mbb{C}$-action on $\mbb{C}^3$ of the form
$t\cdot w=\bigl(w_1,w_2+tf(w_1),w_3+tg(w_1)\bigr)$ where $f(w_1)=-(w_1^2+1)$ 
and $g(w_1)=-2w_1$. Due to a result of Kaliman, see~\cite{K}, every free 
algebraic $\mbb{C}$-action on $\mbb{C}^3$ is conjugate under a biregular map to 
a translation. Hence, it follows that $\mbb{C}_b$ acts properly on $\mbb{C}^3$ 
and thus on $\mbb{C}^4$. Moreover, we have $\mbb{C}^4/\mbb{C}_b\cong\mbb{C}^3$. 
Consequently, Lemma~\ref{Lem:PropernessinSteps} applies to show that 
$N_\Gamma^\mbb{C}$ acts properly on $\Sym(3,\mbb{C})$ with quotient $\mbb{C}^3$.

\begin{rem}
In our case it is possible to give an elementary proof of Kaliman's result which 
applies in a slightly more general setup. For this, let $f,g\in\mathcal{O} 
(\mbb{C})$ be two entire functions and consider the holomorphic 
$\mbb{C}$-action on $\mbb{C}^3$ given by $t\cdot 
z=\bigl(z_1,z_2+tf(z_1),z_3+tg(z_1)\bigr)$. This action is free if and only if 
$f$ and $g$ have no common zeros. In this case, there exist entire functions 
$\varphi,\psi\in\mathcal{O}(\mbb{C})$ such that $\varphi f+\psi g = 1$, 
see~\cite{Hel}. Then the biholomorphic map $\Phi\colon\mbb{C}^3\to\mbb{C}^3$ 
given by
\begin{equation*}
\Phi(z)=
\begin{pmatrix}
1 & 0 & 0\\
0 & f(z_1) & g(z_1)\\
0 & -\psi(z_1) & \varphi(z_1)\\
\end{pmatrix}
\begin{pmatrix}
z_1\\ z_2\\ z_3
\end{pmatrix}
\end{equation*}
transforms the $\mbb{C}$-action into a translation.
\end{rem}

\subsection{Open problems and concluding remarks}

Let $\Gamma\subset N$ be a discrete subgroup with Zariski closure $N_\Gamma$. 
As we have seen, in general the condition that only one $N_\Gamma$-orbit is 
totally real in $D$ does not imply that $D/\Gamma$ is Stein. It is natural to 
ask, however, whether $D/\Gamma$ is Stein if all $N_\Gamma$-orbits in $D$ are 
totally real. In this case, we obtain a free algebraic action of 
$N_\Gamma^\mbb{C}$ on the domain $N_\Gamma^\mbb{C}\cdot\wh{D}\subset\mbb{C}^n$.

As we have observed in the above examples, even if all $N_\Gamma$-orbits are 
totally real in $D$, the hypothesis of Proposition~\ref{Prop:suff} does not 
need to hold true. Then, in order to answer the above question, one could try a 
direct approach similar to the one carried out in the previous subsection. 
This, however, poses two major problems. First of all, one has to show that 
$N_\Gamma^\mbb{C}$ acts properly on 
$N_\Gamma^\mbb{C}\cdot\wh{D}\subset\mbb{C}^n$. Secondly, if $N_\Gamma^\mbb{C}$ 
does indeed act properly, one must prove that 
$\bigl(N_\Gamma^\mbb{C}\cdot\wh{D}\bigr)/N_\Gamma^\mbb{C}$ is a Stein manifold.
These questions are far from being trivial, bearing in mind that there is 
an example of a proper algebraic $\mbb{C}^2$-action on $\mbb{C}^6$ by 
affine-linear transformations such that $\mbb{C}^6/\mbb{C}^2$ is quasi-affine 
but not Stein, see~\cite{Win}.


\begin{thebibliography}{99}

\bibitem{ABKr}
D.~Achab, F.~Betten, and B.~Kr{\"o}tz, \emph{Discrete group actions on {S}tein 
  domains in complex {L}ie groups}, Forum Math. \textbf{16} (2004), no.~1, 
  37--68.

\bibitem{Bo}
A.~Borel, \emph{Linear algebraic groups}, second ed., Graduate Texts in
  Mathematics, vol. 126, Springer-Verlag, New York, 1991.  

\bibitem{BHH}
D.~Burns, S.~Halverscheid, and R.~Hind, \emph{The geometry of {G}rauert tubes
  and complexification of symmetric spaces}, Duke Math. J. \textbf{118} (2003),
  no.~3, 465--491.

\bibitem{BS}
D.~Burns, Jr. and S.~Shnider, \emph{Spherical hypersurfaces in complex
  manifolds}, Invent. Math. \textbf{33} (1976), no.~3, 223--246.

\bibitem{Car}
H.~{C}artan, \emph{Sur les groupes de transformations analytiques},
  (Expos\'es ma\-th\'ematiques IX.) Actual. scient. et industr., vol. 198,
  Hermann, Paris, 1935.

\bibitem{Ch}
B.-Y.~Chen, \emph{Discrete groups and holomorphic functions}, Math. Ann.
  \textbf{355} (2013), no.~3, 1025--1047.
 
\bibitem{DK}
S.~Dey and M.~Kapovich, \emph{A note on complex-hyperbolic Kleinian 
  groups}, Arnold Math. J. \textbf{6} (2020), no.~3-4, 397--406.

\bibitem{DocGr}
F.~Docquier and H.~Grauert, \emph{Levisches {P}roblem und {R}ungescher
  {S}atz f\"ur {T}eilgebiete {S}teinscher {M}annigfaltigkeiten}, Math. Ann.
  \textbf{140} (1960), 94--123.

\bibitem{FabIan}
C.~de~Fabritiis and A.~Iannuzzi, \emph{Quotients of the unit ball of
  {$\mathbb{C}^n$} for a free action of {$\mathbb{Z}$}}, J. Anal. Math.
  \textbf{85} (2001), 213--224.
  
\bibitem{FK}
J.~Faraut and A.~Kor{\'a}nyi, \emph{Analysis on symmetric cones}, Oxford 
  Mathematical Monographs, The Clarendon Press, Oxford University Press, New  
  York, 1994.
  
\bibitem{GH}
B.~Gilligan and A.~T.~Huckleberry, \emph{On non-compact complex nil-manifolds}, 
  Math. Ann. \textbf{238} (1978), no.~1, 39--49.

\bibitem{Gr}
H.~Grauert, \emph{{\"U}ber Modifikationen und exzeptionelle analytische 
  Mengen},  Math. Ann. \textbf{146} (1962), 331--368.

\bibitem{GS}
V.~Guillemin and S.~Sternberg, \emph{Symplectic techniques in physics}, 
  Cambridge University Press, Cambridge, 1984.

\bibitem{HW}
F.~Reese Harvey and R.~O. Wells, Jr., \emph{Zero sets of non-negative strictly
  plurisubharmonic functions}, Math. Ann. \textbf{201} (1973), 165--170.
  
\bibitem{He}
P.~Heinzner, \emph{Geometric invariant theory on {S}tein spaces}, Math. Ann.
  \textbf{289} (1991), no.~4, 631--662.
  
\bibitem{H}
S.~Helgason, \emph{Differential geometry, {L}ie groups, and symmetric
  spaces}, corrected reprint of the 1978 original, Graduate Studies in
  Mathematics~\textbf{34}, American Mathematical Society, Providence, RI, 2001.
  
\bibitem{Hel}
O.~Helmer, \emph{Divisibility properties of integral functions}, Duke Math. J. 
  \textbf{6} (1940), 345--356.

\bibitem{Is}
H.~Ishi, \emph{On symplectic representations of normal {$j$}-algebras and
  their application to {X}u's realizations of {S}iegel domains}, Differential
  Geom. Appl. \textbf{24} (2006), no.~6, 588--612.

\bibitem{K}
S.~Kaliman, \emph{Free ${\bf C}_+$-actions on ${\bf C}^3$ are 
  translations},  Invent. Math. \textbf{156} (2004), no.~1, 163--173.
  
\bibitem{Kan2}
S.~Kaneyuki, \emph{Homogeneuous bounded domains and Siegel domains}, Lecture 
  Notes in Mathematics \textbf{241}, Springer-Verlag, Berlin-New York, 1971.
  
\bibitem{Kob}
S.~Kobayashi, \emph{Hyperbolic complex spaces}, Grundlehren der mathematischen 
  Wissenschaften \textbf{318}, Springer-Verlag, Berlin, 1998.

\bibitem{Kos}
J.~L.~Koszul, \emph{Sur la forme hermitienne canonique des espaces 
  homog{\`e}nes complexes}, Canadian J. Math. \textbf{7} (1955), 562--576.

\bibitem{L}
J.-J.~Loeb, \emph{Action d'une forme r\'eelle d'un groupe de {L}ie
  complexe sur les fonctions plurisousharmoniques}, Ann. Inst. Fourier
  (Grenoble) \textbf{35} (1985), no.~4, 59--97.

\bibitem{Mie4}
C.~Miebach, \emph{Quotients of bounded homogeneous domains by cyclic
  groups}, Osaka J. Math. \textbf{47} (2010), no.~2, 331--352.

\bibitem{Mie5}
C.~Miebach, \emph{Sur les quotients discrets de semi-groupes complexes}, Ann. 
  Fac. Sci. Toulouse Math. (6) \textbf{19} (2010), no.~2, 269--276.

\bibitem{MO}
C.~Miebach and K.~Oeljeklaus, \emph{On proper {$\mathbb R$}-actions on
  hyperbolic {S}tein surfaces}, Doc. Math. \textbf{14} (2009), 673--689.

\bibitem{MO2}
C.~Miebach and K.~Oeljeklaus, \emph{Schottky group actions in complex geometry}, 
  Geometric complex analysis, Springer Proc. Math. Stat., vol. 246, Springer, 
  Singapore, 2018, pp.~257--268.
  
\bibitem{Pal}
R.~S.~Palais, \emph{A global formulation of the {L}ie theory of transformation 
  groups}, Mem. Amer. Math. Soc. No.~\textbf{22} (1957), p.~iii+123.
  
\bibitem{P}
R.~S.~Palais, \emph{On the existence of slices for actions of non-compact Lie 
  groups},  Ann. of Math. (2) \textbf{73} (1961), 295--323.
  
\bibitem{Pya}
I.~I.~Pyateskii-Shapiro, \emph{Automorphic functions and the geometry of
  classical domains}, Translated from the Russian. Mathematics and Its
  Applications, Vol. 8, Gordon and Breach Science Publishers, New York, 1969.
  
\bibitem{R}
H.~L.~Royden, \emph{Holomorphic fiber bundles with hyperbolic fiber}, Proc.
  Amer. Math. Soc. \textbf{43} (1974), 311--312.

\bibitem{Vin4}
{\`E}.~B. Vinberg, \emph{The {M}orozov-{B}orel theorem for real {L}ie groups},
  Dokl. Akad. Nauk SSSR \textbf{141} (1961), 270--273.

\bibitem{Vin3}
{\`E}.~B. Vinberg (ed.), \emph{Lie groups and {L}ie algebras, {III}},
  Encyclopaedia of Mathematical Sciences, vol.~41, Springer-Verlag, Berlin,
  1994.
 
\bibitem{VGP}
{\`E}.~B. Vinberg, S.~G.~Gindikin, and I.~I.~Pyateskii-Shapiro,
  \emph{Classification and canonical realization of complex homogeneous bounded 
  domains}, Trudy Moskov. Mat. Ob\v{s}\v{c}. \textbf{12} (1963), 359--388.

\bibitem{Vi}
S.~Vitali, \emph{Quotients of the crown domain by a proper action of a cyclic
  group}, Trans. Amer. Math. Soc. \textbf{366} (2014), no.~6, 3227--3239.
  
\bibitem{Win}
J.~Winkelmann, \emph{On free holomorphic {$\bf C$}-actions on {${\bf
  C}\sp n$} and homogeneous {S}tein manifolds}, Math. Ann. \textbf{286} (1990),
  no.~1-3, 593--612.

\end{thebibliography}
\end{document}